\documentclass[11pt,a4paper]{amsart}
\usepackage[utf8]{inputenc}
\usepackage{graphicx,epsfig}
\usepackage{amsmath, amsthm, amssymb,enumerate,pinlabel}
\newtheorem{theorem}{Theorem}[section]
\newtheorem{prop}[theorem]{Proposition}
\newtheorem{prop*}{Proposition}

\newtheorem{cor}[theorem]{Corollary}
\newtheorem{cor*}{Corollary}
\newtheorem{theorem*}{Theorem }
\theoremstyle{definition}
\newtheorem{defn}[theorem]{Definition}

\newtheorem{exmp*}{Example}
\newcommand{\Mod}{\mathrm{Mod}}
\newcommand{\Fix}{\text{Fix}}
\newcommand{\Teich}{\mathrm{Teich}}
\newcommand{\Homeo}{\mathrm{Homeo}}
\newcommand{\M}{\mathcal{M}}
\newcommand{\B}{\mathcal{B}}

\newcommand{\F}{\mathcal{F}}
\newcommand{\Z}{\mathbb{Z}}
\renewcommand{\P}{\mathcal{P}}
\renewcommand{\O}{\mathcal{O}}

\begin{document}

\title[Estimating distances between hyperbolic structures]{Estimating the distances between \\ hyperbolic structures in the moduli space}

\author{Atreyee Bhattacharya}
\address{Department of Mathematics\\
Indian Institute of Science Education and Research Bhopal\\
Bhopal Bypass Road, Bhauri \\
Bhopal 462066, Madhya Pradesh\\
India}
\email{atreyee@iiserb.ac.in}

\author{Suman Paul}
\email{smnpl2009@gmail.com}

\author{Kashyap Rajeevsarathy}
\address{Department of Mathematics\\
Indian Institute of Science Education and Research Bhopal\\
Bhopal Bypass Road, Bhauri \\
Bhopal 462066, Madhya Pradesh\\
India}
\email{kashyap@iiserb.ac.in}
\urladdr{https://home.iiserb.ac.in/$_{\widetilde{\phantom{n}}}$kashyap/}



\maketitle
\begin{abstract}
Let $\mathrm{Mod}(S_g)$ be the mapping class group of the closed orientable surface $S_g$ of genus $g\geq 2$. Given a finite subgroup $H$ of $\mathrm{Mod}(S_g)$, let $\mathrm{Fix}(H)$ be the set of all fixed points induced by the action of $H$ on the Teichm\"{u}ller space $\mathrm{Teich}(S_g)$ of $S_g$. This paper provides a method to estimate the distance between the unique fixed points of certain irreducible cyclic actions on $S_g$. We begin by deriving an explicit description of a pants decomposition of $S_g$, the length of whose curves are bounded above by the Bers' constant. To obtain the estimate, our method then uses the quasi-isometry between $\mathrm{Teich}(S_g)$ and the pants graph $\mathcal{P}(S_g)$. 
\end{abstract}
\section{Introduction}
Let $S_g$ denote the closed orientable surface of genus $g \geq 2$, $\Mod(S_g)$ the mapping class group of $S_g$, $\Teich(S_g)$ the Teichm\"{u}ller space of $S_g$, and $\M(S_g)$ the moduli space of $S_g$. Given a finite subgroup $H \leq \Mod(S_g)$, let $\Fix(H)$ denote the set of fixed points induced by the natural action of $H$ on $\Teich(S_g)$. The Nielsen realization problem asks whether $\Fix(H) \neq \emptyset$, for an arbitrary finite subgroup $H \leq \Mod(S_g)$. While S. Kerckhoff~\cite{SK1} settled this in the affirmative, W. Harvey~\cite[Theorem 2]{H1} showed that $\Fix(H) \approx \hat{i} (\Teich(S_g/H))$, where $\Teich(S_g/H)$ is defined in the sense of L. Bers~\cite{LB1,LB2}, and $\hat{i}$ is the natural embedding induced by the branched cover $S_g \to S_g/H$ (identifying $H$ with a group of self-homeomorphisms of $S_g$). However, the results of Kerkhoff and Harvey did not provide a concrete description of $\Fix(H)$. Only recently in~\cite{BPK} an explicit parametrization of $\Fix(H)$ has been obtained as a  K\"ahler submanifold of $\Teich(S_g)$ where $H$ is an arbitrary finite cyclic subgroup of $\Mod(S_g)$ and $\Teich(S_g)$ is viewed as the K\"ahler manifold equipped with the Weil-Petersson metric $\mu_{wp}$. This gives a large class of totally geodesic submanifolds of $(\Teich(S_g),\mu_{wp})$. As a natural follow up of this, in this article we provide estimates on distances between some of these submanifolds of $(\Teich(S_g),\mu_{wp})$ arising as sets of fixed points of certain finite cyclic subgroups of $\Mod(S_g)$. This is an important pursuit as although $(\Teich(S_g),\mu_{wp})$ is known to be geodesically convex, providing explicit descriptions of geodesics between hyperbolic structures on $S_g$ is extremely hard. 
 
 It is a well-known result of Brock \cite{JB} that $\Teich(S_g)$ endowed with the Weil-Petersson metric $\mu_{wp}$ is quasi-isometric to the pants graph $\P(S_g)$ of $S_g$. In this paper, we use this quasi-isometry and the theory developed in~\cite{BPK,PKS} to describe a method (see Sections \ref{Section:Admissible_Pants Decomposition} and \ref{Section:distances_Pants graph}) to estimate the distances between submanifolds of $\Teich(S_g)$ mentioned above. 

Given a simple closed curve $\alpha$ in $S_g$ and an $X \in \Teich(S_g)$, let $\ell_X(\alpha)$ denote the length of the geodesic representative of $\alpha$ under the hyperbolic metric in $X$. Bers~\cite{LB1,LB2} showed the existence of a universal constant $\B_g$ (known as the \textit{Bers' constant}) depending only on $g$ such that for any $X \in \Teich(S_g)$, there exists a pants decomposition $P_X$ of $S_g$ such that $\ell_X(\alpha) < \B_g$ for each $\alpha \in P_X$. For a pants decomposition $P$ of $S_g$, let $$V(P) = \{X \in \Teich(S_g) : \max_{\alpha \in P} \ell_X(\alpha) < \B_g \}.$$ Thus, a direct consequence of Bers' result is that the collection $\{ V(P) \} $ forms a basis for the topology on $(\Teich(S_g), \mu_{wp})$. Furthermore, given $X \in V(P_X)$ and $Y \in V(P_Y)$ the quasi-isometry from~\cite{JB} allows us to estimate $\mu_{wp}(X,Y)$ with the distance between $P_X$ and $P_Y$ in $P(S_g)$.

An $h \in \Mod(S_g)$ of order $n$ is said to be \textit{Type 1 irreducible} if its corresponding orbifold a sphere with three cone points where least one cone point is of order $n$. In~\cite{BPK,PKS}, it was shown that an irreducible Type 1 mapping class $h$ is geometrically realized as the rotation of a unique hyperbolic semi-regular polygon $X_h$ with a certain side-pairing. Moreover, an inductive procedure was described to build an arbitrary periodic mapping class from irreducible Type 1 components through certain processes called \textit{compatibilities}. Consequently, the hyperbolic structures realizing arbitrary periodic maps on $S_g$ can also be constructed inductively using the structures of type $X_h$. In this paper, the first main result describes an explicit pants decomposition $P$ of $S_g$ such that $P = P_{X_h}$ (see Theorem~\ref{thm:admiss_pants_decomp}). The second main result of this article provides an estimate for the distance between sets of fixed points of two Type 1 irreducible actions as follows.
\begin{theorem*} \label{intro:main1}
Let $H_1=\langle h_1 \rangle$ and $H_2=\langle h_2\rangle$ be finite cyclic subgroups of $\Mod(S_g)$ where $h_1$ and $h_2$ are irreducible type 1 mapping classes. Then
$$d_{wp}(\Fix(H_1), \Fix(H_2))\leq K\mathbb{D} (h_1,h_2)+\epsilon,$$
where $K$ and $\epsilon$ are the quasi-isometric constants from Brock's result (see \cite{JB}), $d_{WP}$ is the distance function induced by the Weil-Petersson metric $\mu_{wp}$, and $\mathbb{D}(h_1,h_2)$ is a global constant completely determined by the polygons $X_{h_1}$ and $X_{h_2}$.
\end{theorem*}

\section{Preliminaries}
\subsection{Periodic mapping classes}\label{sec:finite_maps}
For $g \geq 1$, let $F \in \Mod(S_g)$ be of order $n$. By the Nielsen-Kerckhoff theorem~\cite{SK,JN},  $F$ is represented by a \textit{standard representative} $\F \in \Homeo^+(S_g)$ of order $n$. Let $\O_F := S_g/\langle \F \rangle$ be the \textit{corresponding orbifold} of $F$ of genus $g_0$ (say). Each cone point $x_i \in \O_F$ lifts under the branched cover $S_g \to S_g/\langle \F \rangle$ to an orbit of size $n/n_i$ on $S_g$, where the local rotation induced by $\F$ is given by $2 \pi c_i^{-1}/n_i$,  $c_i c_i^{-1} \equiv 1 \pmod{n_i}$. The tuple $\Gamma(\O_F) := (g_0; n_1,\ldots,n_{\ell})$, is called the \textit{signature of} $\O_F$. By the theory of group actions on surfaces (see~\cite{H1} and the references therein), we obtain an exact sequence: 
\begin{equation*}
\label{eq:surf_kern}
1 \rightarrow \pi_1(S_g) \rightarrow \pi_1^{orb}(\O_F) \xrightarrow{\Phi_F} \langle \F \rangle \rightarrow 1, 
\end{equation*}
\noindent where the orbifold fundamental group $\pi_1^{orb}(\O_F)$ of $\O_F$ has the presentation
\begin{equation*}
\label{eqn:orb-pres}
\left\langle \alpha_1,\beta_1,\dots,\alpha_{g_0},\beta_{g_0}, \xi_1,\dots,\xi_{\ell} \, |\, \xi_1^{n_1},\dots,\xi_\ell^{n_{\ell}},\,\prod_{j=1}^{\ell} \xi_j \prod_{i=1}^{g_0}[\alpha_i,\beta_i]\right\rangle
\end{equation*} 
and $\Phi_{F} (\xi _i) = \F^{(n/n_i)c_i}$, for $1 \leq i \leq \ell$. We will now define a tuple of integers that will encode the conjugacy class of a periodic mapping class $F \in \Mod(S_g)$ of order $n$ in $\Mod(S_g)$.  

\begin{defn}\label{defn:data_set}
A \textit{data set of degree $n$} is a tuple
$$
D = (n,g_0, r; (c_1,n_1),\ldots, (c_{\ell},n_{\ell})),
$$
where $n\geq 2$, $g_0 \geq 0$, and $0 \leq r \leq n-1$ are integers, and each 
$c_i \in \Z_{n_i}^\times$ such that:
\begin{enumerate}[(i)]
\item $r > 0$ if and only if $\ell = 0$ and $\gcd(r,n) = 1$, whenever $r >0$,
\item each $n_i\mid n$,
\item $\text{lcm}(n_1,\ldots \widehat{n_i}, \ldots,n_{\ell}) = N$, for $1 \leq i 
\leq \ell$, where $N = n$, if $g_0 = 0$,  and
\item $\displaystyle \sum_{j=1}^{\ell} \frac{n}{n_j}c_j \equiv 0\pmod{n}$.
\end{enumerate}
The number $g$ determined by the Riemann-Hurwitz equation
\[\frac{2-2g}{n} = 2-2g_0 + \sum_{j=1}^{\ell} \left(\frac{1}{n_j} - 1 \right) \tag{R-H} \]
is called the {genus} of the data set, denoted by $g(D)$.
\end{defn}

\noindent The following proposition, which allows us to use data sets to represent the conjugacy classes of cyclic actions on $S_g$, follows mainly from a result of Nielsen~\cite{JN1}.

\begin{prop}\label{prop:ds-action}
For $g \geq 1$ and $n \geq 2$, data sets of degree $n$ and genus $g$ correspond to conjugacy classes of $\Z_n$-actions on $S_g$. 
\end{prop}

\noindent We will denote the data set encoding the conjugacy class of a periodic mapping class $F$ by $D_F$. The parameter $r$ (in Definition~\ref{defn:data_set}) will come into play in a data set $D_F$ only when $\F$ is a free rotation of $S_g$ by $2\pi r/n$, in which case, $D_F$ will take the form $(n,g_0,r;)$. We will avoid including $r$ in the notation of a data set $D_F$, whenever $\F$ is non-free.  Furthermore, for compactness of notation, we also write a data set $D$ (as in Definition~\ref{defn:data_set}) as
$$D = (n,g_0,r; ((d_1,m_1),\alpha_1),\ldots,((d_{\ell'},m_{\ell'}),\alpha_{\ell'})),$$
where $(d_i,m_i)$ are the distinct pairs in the multiset $S = \{(c_1,n_1),\ldots,(c_{\ell},n_{\ell})\}$, and the $\alpha_i$ denote the multiplicity of the pair $(d_i,m_i)$ in the multiset $S = \{(c_1,n_1),\ldots,(c_{\ell},n_{\ell})\}$. For simplicity of notation, we shall avoid writing the parameters $\alpha_i$ when they equal $1$. 

Let $F \in \Mod(S_g)$ be of order $n$. Then $F$ is said to be \textit{rotational} if $\F$ is a rotation of the $S_g$ through an axis by $2 \pi r/n$, where $\gcd(r,n)=1$. It is apparent that $\F$ is either has no fixed points, or $2k$ fixed points which are induced at the points of intersection of the axis of rotation with $S_g$. Moreover, these fixed points will form $k$ pairs of points $(x_i,x_i')$, for $1 \leq i \leq k$, such that the sum of the angles of rotation induced by $\F$ around $x_i$ and  $x_i'$ add up to $0$ modulo $2\pi$. Consequently, we have the following: 
\begin{prop}
\label{prop:rotl_actn}
Let $F \in \Mod(S_g)$ be a rotational mapping class of order $n$. 
\begin{enumerate}[(i)]
\item When $\F$ is a non-free rotation, then $D_F$ has the form
 $$(n,g_0;\underbrace{(s,n),(n-s,n),\ldots,(s,n),(n-s,n)}_{k \,pairs}),$$ for integers $k \geq 1$ and $0<s\leq n-1$ with $\gcd(s,n)= 1$, and $k=1$ if $n>2$.  
 \item When $\F$ is a free rotation, then $D_F$ has the form
 $$(n,\frac{g-1}{n}+1,r;).$$
\end{enumerate}
\end{prop}

\noindent We say $F$ is of \textit{Type 1} if $\Gamma(\O_F)$ has the form $(g_0; n_1,n_2,n)$, and $F$ is said to be of \textit{Type 2} if $F$ is neither rotational nor of Type 1. Gilman~\cite{G1} showed that a periodic mapping class $F \in \Mod(S_g)$ is irreducible if and only if $\O_F$ is a sphere with three cone points. Thus, $F$ is an irreducible Type 1 mapping class if and only if $\Gamma(\O_F)$ has the form $(0; n_1,n_2,n)$.

\subsection{Decomposing periodic maps into irreducibles} 
\label{sec:pri_into_irred}
In ~\cite{BPK,PKS}, a method was described to construct an arbitrary non-rotational periodic element $F \in \Mod(S_g)$, for $g \geq 2$, by performing certain``compatibilties" on irreducible Type 1 actions. These Type 1 actions are in turn realized as rotations of certain unique hyperbolic polygons with side-pairings. 
\begin{theorem}[{\cite[Theorem 2.7]{PKS}}]
\label{res:1}
For $g \geq 2$, consider an irreducible Type 1 action $F \in {\Mod}(S_g)$ with $$D_F = (n,0; (c_1,n_1),\linebreak (c_2,n_2), (c_3,n)).$$ Then $F$ can be realized explicitly as the rotation $\theta_F= 2\pi c_3^{-1}/n$ of a hyperbolic polygon $\P_F$ with a suitable side-pairing $W(\P_F)$, where $\P_F$ is a hyperbolic  $k(F)$-gon with
$$ k(F) := \begin{cases}
2n, & \text { if } n_1,n_2 \neq 2, \text{ and } \\
n, & \text{otherwise, }
\end{cases}$$
and for $0 \leq m\leq n-1$, 
$$ 
W(\P_F) =
\begin{cases}
\displaystyle  
  \prod_{i=1}^{n} a_{2i-1} a_{2i} \text{ with } a_{2m+1}^{-1}\sim a_{2z}, & \text{if } k(F) = 2n, \text{ and } \\
\displaystyle
 \prod_{i=1}^{n} a_{i} \text{ with } a_{m+1}^{-1}\sim a_{z}, & \text{otherwise,}
\end{cases}$$
where $\displaystyle z \equiv m+qj \pmod{n}$ with $q= (n/n_2)c_3^{-1}$ and $j=n_{2}-c_{2}$.
\end{theorem} 
\subsection{The Teichm\"uller space of $S_g$}
Let $\text{HypMet}(S_g)$ denote the set of all hyperbolic metrics on $S_g$ and $\text{Diff}_0(S_g)$ denote the group of all diffeomorphisms of $S_g$ isotopic to identity. The Teichm\"uller space ($\Teich(S_g)$) of $S_g$ is the quotient space 
$$\Teich(S_g) = \text{HypMet}(S_g)/\text{Diff}_0(S_g)$$
where $\text{Diff}_0(S_g)$ ($\text{Diffeo}^{+}(S_g)$) acts on $\text{HypMet}(S_g)$ via \textit{pullback} i.e. 
$$f \cdot \xi = f^{*}(\xi) \ \forall f \in \text{Diff}_0(S_g) \ \text{ and } \xi \in \text{HypMet}(S_g).$$ 
The above action induces a natural action of $\Mod(S_g)$ on $\Teich(S_g)$ as follows:
Given $F = [f]\in \Mod(S_g)$ and  $[\xi]\in \Teich(S_g)$, 
$$F\cdot [\xi] = [f^{*}(\xi)].$$
\subsection{The pants graph}
A pants decomposition of $S_g$ ($g\geq 2$) is a maximal multicurve comprising $3g-3$ distinct isotopy classes of pairwise disjoint, essential, simple closed curves on $S_g$. Two given distinct pants decompositions $P_1$ and $P_2$ of $S_g$ are said to be related by an elementary move provided $P_2$ can be obtained from $P_1$ by replacing a curve $\alpha \in P_1$  by a curve $\beta$ that  intersects $\alpha$ minimally. Considering each pants decomposition as a vertex and joining two pants decompositions by an edge if they differ by an elementary move, one obtains the pants graph $\P(S_g)$ of $S_g$. Moreover, $\P(S_g)$ is a metric space by assigning length $1$ to each of its edges. Brock proved that  $(\Teich(S_g), \mu_{wp})$ is quasi-isometric to the pants graph $\P(S_g)$ equipped with the aforementioned metric. This enables us to study the coarse geometric aspects of the Weil-Petersson metric via the combinatorial description of $\P(S_g)$. The quasi-isometry can be described as follows:  Given a hyperbolic metric $X\in \Teich(S_g)$ and a simple closed curve $\alpha$ on $S_g$, let $\ell_X(\alpha)$ denote the length of the geodesic representative of $\alpha$ with respect to the hyperbolic metric $X$. By a result of Bers (see \cite{LB1,LB2}), it follows that there is a universal constant $\B_g$ (known as the Bers’ constant) depending only on $g$ such that for each $X \in\Teich(S_g)$, there exists a pants decomposition $P_X$ of $S_g$ such that $\ell_X(\alpha) < \B_g$ for each $\alpha \in P_X$.  

\section{Finding an Admissible Pants Decomposition for $P_{\kappa(D)}$}\label{Section:Admissible_Pants Decomposition}
\subsection{Admissible Pants Decomposition}
For $g \geq 2$, let $F \in {\Mod}(S_g)$ be an irreducible Type 1 action with 
$$D_F = (n,0; (c_1,n_1),\linebreak (c_2,n_2), (c_3,n))$$ 
realized explicitly as the rotation $\theta_F= 2\pi c_3^{-1}/n$ of a unique hyperbolic $k(F)$-gon $\P_F$ as mentioned in Theorem \ref{res:1}. In this section, we develop an algorithm in order to construct a special pants decomposition of $S_g$ that is \textit{admissible} for the hyperbolic structure $\P_{F}$ in the following sense.

Here are some remarks about the presentation of polygon $\mathcal{P}_F$, when $k(F)=2n$. This technical aspect will be used later.
\begin{defn}
A pants decomposition $\mathcal{P}= \{\alpha_1,\alpha_2,\cdots,\alpha_{3g-3}\}$ of $S_g$ is said to be \textit{admissible for the hyperbolic structure} $\P_{F}$ if $\ell(\alpha_j)\leq \B_g$ for all $1\leq j \leq 3g-3$ where $\ell(\alpha_j)$ denotes the length of $\alpha_j$ with respect to $\P_{F}$ and $\B_g$ denotes the Bers' constant corresponding to genus $g$.       
\end{defn}

Let $\iota(a_k)$ and $t(a_k)$ denote the initial and terminal vertices of the edge $a_k$ in the hyperbolic polygon $\P_{F}$ described above. Let $\gamma_i$ denote the geodesic joining vertices $\iota(a_i)$ and $\iota(a_{2n+1-i})$, and $\tilde{\gamma}_i$ the geodesic joining the mid-points of $a_i$ and $a_{2n+1-i}$. We consider the multicurve $\mathcal{C}$ given by 
$$\mathcal{C}=\{\gamma_i\}_i\cup \{\tilde{\gamma}_i\}_i\cup \{\gamma_i\ast\gamma_j\}_{i,j}\cup \{\tilde{\gamma}_i\ast \tilde{\gamma}_j\}_{i,j}$$
as shown in Figure~\ref{fig:poly_part} below. 

\begin{figure}[htbp]
\centering
\labellist
	\small
	\pinlabel $a_1$ at 50 4
	\pinlabel $a_2$ at 85 18
	\pinlabel $a_3$ at 98 52
	\pinlabel $a_4$ at 84 87
     \pinlabel $a_1$ at 50 102
     \pinlabel $a_2$ at 14 87
     \pinlabel $a_4$ at 15 18
     \pinlabel $a_3$ at 1 52
     \pinlabel $a_1$ at 158 2
     \pinlabel $a_1$ at 156 103
      \pinlabel $a_2$ at 190 13
      \pinlabel $a_2$ at 125 93
      \pinlabel $a_3$ at 206 37
      \pinlabel $a_3$ at 108 67
       \pinlabel $a_4$ at 110 37
       \pinlabel $a_4$ at 205 67
      \pinlabel $a_5$ at 185 93
       \pinlabel $a_5$ at 127 13
       \pinlabel $R_1$ at 190 29
       \pinlabel $R_2$ at 175 40
       \pinlabel $R_3$ at 157 54
        \pinlabel $R_4$ at 140 65
         \pinlabel $R_5$ at 124 78
          \pinlabel $R_1$ at 85 38
         \pinlabel $R_2$ at 64 48
         \pinlabel $R_3$ at 37 59
         \pinlabel $R_4$ at 14 68
\endlabellist
\includegraphics[width=55ex]{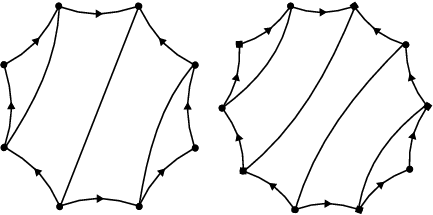}
\caption{Partitioning the hyperbolic structures on $S_2$ realizing the irreducible Type 1 actions of orders $8$ and $10$.}
\label{fig:poly_part}
\end{figure}
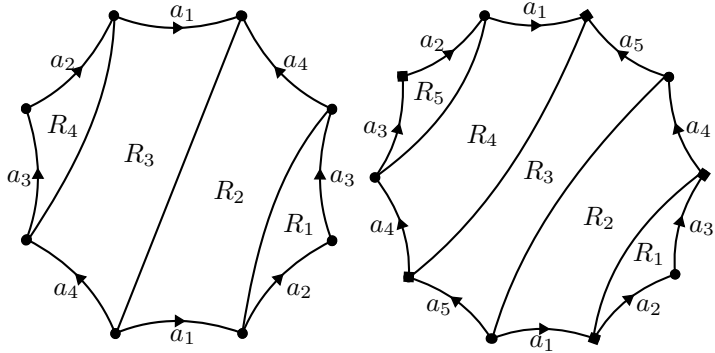

Note that, each $\gamma_i$ (or $\tilde{\gamma}_i$) in the above list need not be a closed curve. However, we observe the following. 
\begin{prop}\label{prop:pants}
A subcollection of $\mathcal{C}$ consisting of only simply closed curves forms a pants decomposition of $S_g$.
\end{prop}
\begin{proof}
It suffices to show that, upon cutting along the curves $\gamma_i$'s and $\tilde{\gamma}_i$'s, $\P_{F}$ breaks into cylinders and pair of pants. Let $R_1$ and $R_n$ denote the triangles $\Delta a_1 \gamma_1 a_{2n}$ and $\Delta a_n a_{n+1} \gamma_{n-1}$ respectively. For $i=2,\cdots,n-1$, let $R_i$ denote the quadrilateral bounded by $a_i, \gamma_i, a_{2n+1-i}$ and $\gamma_{i-1}$. Thus by construction, each $R_i$ can be uniquely represented by the pair $(a_i,a_{2n+1-i})$. We define $R_{i}^{-1}:=(a_{i}^{-1},a_{2n+1-i}^{-1})$. 
We claim that, if $a_{i}^{-1}\sim a_j$ then $a_{2n+1-i}^{-1}\sim a_{2n+1-j}$. To prove the claim, we consider two cases depending on whether $i$ is odd or even.

\textit{Case I: Suppose that $i=2k+1$, for some $k$}. Then from the definition of $\P_{F}$, it follows that $a_{i}^{-1}=a_{2k+1}^{-1}\sim a_{2(k+qj)}$. Therefore, we have $j \equiv 2(k+qj) \pmod{2n}$, which implies that $2n+1-j \equiv 2(n-k-qj)+1 \pmod{2n}$. Hence, it follows that $a_{2n+1-j}^{-1}\sim a_{2(n-k-qj+qj)}=a_{2n-2k}=a_{2n+1-i}$, which proves the claim.

\textit{Case II: Suppose that $i=2k$, for some $k$}. Then $j \equiv 2(k-qj)+1 \mathrm{2n}$,  from which it follows that 
$$a_{2n+1-i}^{-1}=a_{2(n-k)+1)}^{-1}\sim a_{2(n-k+qj)}=a_{2n+1-j},$$ thereby settling the claim.

Finally, we observe that if $a_{i}^{-1}\sim a_{2n+1-i}$, then $R_i$ forms a cylinder. Otherwise, suppose $a_{i}^{-1}\sim a_j$. Then from the previous claim, $a_{2n+1-i}^{-1}\sim a_{2n+1-j}$. Hence, $R_{i}^{-1}=R_j$ and $R_i\cup R_{i}^{-1}$, with the given side-pairing relations, will form a cylinder, a pair of pants, or a four holed sphere, depending on the vertices of the corresponding sides. This is illustrated in Figure \ref{fig:poly_part} below.

\begin{figure}[htbp]
\centering
\labellist
	\small
	\pinlabel $\gamma_1$ at 85 220
	\pinlabel $\gamma_2$ at 160 138
	\pinlabel $R_1$ at 85 250
	\pinlabel $R_1^{-1}$ at 175 250
	\pinlabel $R_1$ at 335 250
	\pinlabel $R_1^{-1}$ at 440 250
	\pinlabel $R_1$ at 585 250
	\pinlabel $R_1^{-1}$ at 690 250
	\pinlabel $a_1$ at 40 170
	\pinlabel $a_1$ at 210 175
	\pinlabel $a_2$ at 98 190
	\pinlabel $a_2$ at 156 170
	\pinlabel $\gamma_1$ at 335 220
	\pinlabel $\gamma_2$ at 435 138
	\pinlabel $\gamma_3$ at 435 200
	\pinlabel $a_1$ at 300 170
	\pinlabel $a_1$ at 495 175
	\pinlabel $a_2$ at 342 190
	\pinlabel $a_2$ at 406 170
	\pinlabel $\gamma_1$ at 594 220
	\pinlabel $\gamma_2$ at 694 138
	\pinlabel $\gamma_3$ at 684 200
	\pinlabel $a_1$ at 559 170
	\pinlabel $a_1$ at 750 175
	\pinlabel $a_2$ at 595 190
	\pinlabel $a_2$ at 660 170
	\pinlabel $R_i$ at 80 -8
	\pinlabel $R_i^{-1}$ at 190 5
	\pinlabel $b_j$ at 20 60
	\pinlabel $c_j$ at 106 60
	\pinlabel $b_j$ at 250 60
	\pinlabel $c_j$ at 163 70
	\pinlabel $b_j$ at 274 60
	\pinlabel $c_j$ at 360 60
	\pinlabel $b_j$ at 498 60
	\pinlabel $c_j$ at 415 70
	\pinlabel $b_j$ at 520 60
	\pinlabel $c_j$ at 608 60
	\pinlabel $b_j$ at 747 60
    \pinlabel $c_j$ at 660 70
	\pinlabel $R_i$ at 350 -8
	\pinlabel $R_i^{-1}$ at 443 0
	\pinlabel $R_i$ at 590 -8
	\pinlabel $R_i^{-1}$ at 685 0
	\endlabellist
	\includegraphics[width=70ex]{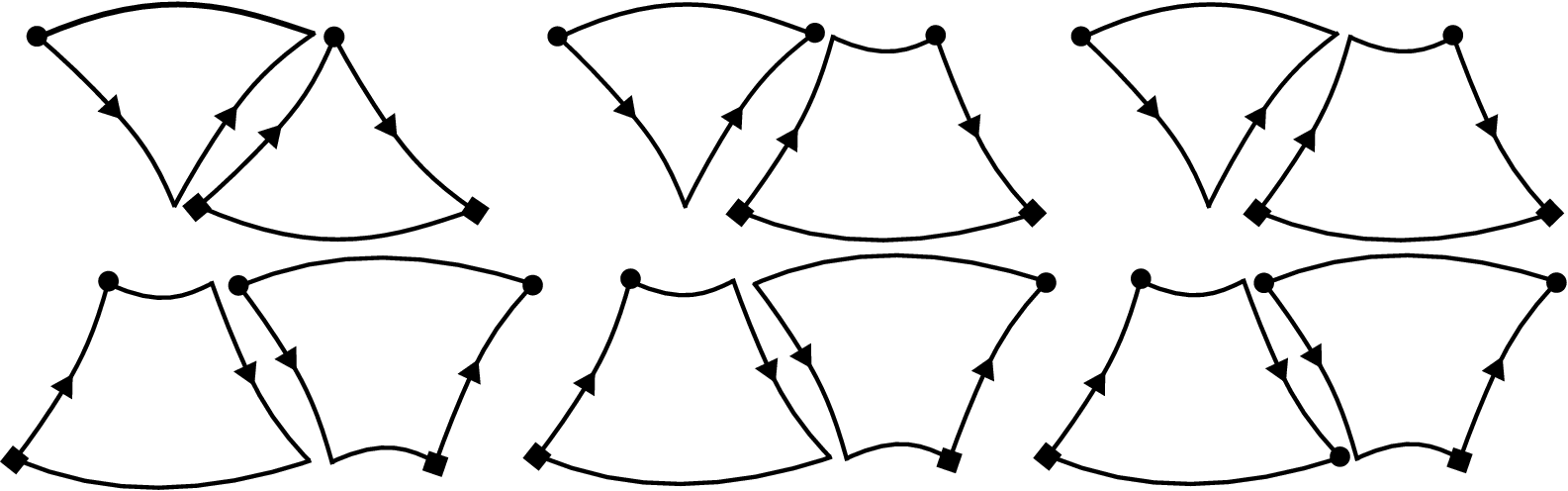}
\caption{Identifying the partitioned pieces of the polygon.}
\label{fig:poly_pieces}
\end{figure}
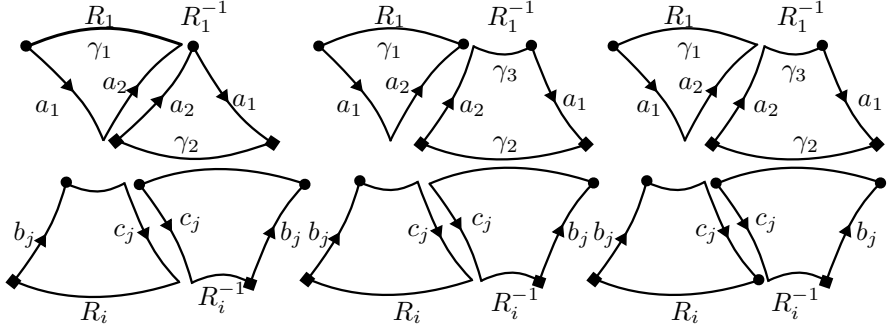

 It is easy to see that in the case of the cylinder and the pair of pants, the pants curves belong to the collection $\{\gamma_i,\gamma_j, \gamma_{i-1},\gamma_{j-1}, \gamma_i\ast \gamma_j , \gamma_i\ast \gamma_{j-1}, \gamma_{i-1}\ast \gamma_j, \gamma_{i-1}\ast \gamma_{j-1}\}$. Furthermore, in the case of a four holed sphere, the pants curves come from  $\{\gamma_i,\gamma_j, \gamma_{i-1},\gamma_{j-1}, \tilde{\gamma}_i\ast \tilde{\gamma}_j\}$. Thus, in all cases, the pants curves belong to the multicurve $\mathcal{C}$. 
\end{proof}

\noindent As immediate consequences of Proposition \ref{prop:pants}, we observe the following:
\begin{cor}
Given a polygon $\P_F$ as in Proposition \ref{prop:pants} and denoting $v_i=\iota(a_i)$, $w_i=t(a_i)$, $\hat{v}_i=t(a_{2n+1-i})$ and $\hat{w}_i=\iota(a_{2n+1-i})$, it follows that, for each $i=2,\ldots,n-1$, these are the four vertices of the quadrilateral $R_i$ as mentioned above. Since $R_1$ and $R_n$ are triangles, we have $v_1=\hat{v}_1$ and $w_n=\hat{w}_n$. Moreover, $w_i=v_{i+1}$ and $\hat{w}_i=\hat{v}_{i+1}$, for $i=1,\ldots, n-1$. With this, one can summarize the descriptions of $R_i\cup R_i^{-1}$ as follows.
\begin{enumerate}[(i)]
\item For $i=1$, we have the following possibilities for $R_1\cup R_{1}^{-1}$. 
\begin{enumerate}[(a)]
\item If $R_{1}^{-1}=R_n$, then $R_1\cup R_{1}^{-1}$ forms a cylinder (see Fig. \ref{fig:poly_pieces}, top left).
\item If $w_1=\hat{w}_1$, then $R_1\cup R_{1}^{-1}$ forms a pair of pants (see Fig. \ref{fig:poly_pieces}, top middle).
\item If $w_1\neq\hat{w}_1$, then $R_1\cup R_{1}^{-1}$ forms a cylinder (see Fig. \ref{fig:poly_pieces}, top right).
\item If $w_i=\hat{w}_i$ and $v_i\neq\hat{v}_i$ (or $v_i=\hat{v}_i$ and $w_i\neq\hat{w}_i$), then $R_i\cup R_{i}^{-1}$ forms a pair of pants (see Fig. \ref{fig:poly_pieces}, bottom left).
\item If $w_1\neq\hat{w}_1$ and $v_i\neq\hat{v}_i$, then $R_i\cup R_{i}^{-1}$ forms a cylinder (see Fig \ref{fig:poly_pieces}, bottom middle).
\item If $w_1=\hat{w}_1$ and $v_i=\hat{v}_i$, then $R_i\cup R_{i}^{-1}$ forms a four holed sphere (see Fig. \ref{fig:poly_pieces}, bottom right).
\end{enumerate}

\item Moreover, the following two conditions provide the necessary and sufficient criteria for a general description of $R_i\cup R_{i}^{-1}$.

\begin{enumerate}[(a)]
\item $v_{2k+1}=\hat{v}_{2k+1}$ if and only if $\gcd(n,qj)|2k$.
\item $w_{2k+1}=\hat{w}_{2k+1}$ if and only if $\gcd(n,qj-1)|2k+1$.
\end{enumerate}
\end{enumerate}
\end{cor}
\begin{proof}
It suffices to prove Case (ii), which is the more general case.
\begin{enumerate}[(a)]
\item It follows that $v_{2k+1}=\hat{v}_{2k+1}$
     $\iff \iota(a_{2k+1})=t(a_{2(n-k)})$. Since $\iota(a_{2k+1})=t(a_{2k+1}^{-1})=t(a_{2(k+qj)})=a_{2(k+mqj)}$, we have:
     \begin{eqnarray*}
     v_{2k+1}=\hat{v}_{2k+1} & \iff & 2(n-k)=2(k+mqj)(\textrm{mod }2n), \text{where } m\in \mathbb{N}\\
     &\iff & n-k \equiv k+mj \pmod{n}\\
     &\iff &k+mj=nM+(n-k), \text{where } M \in\mathbb{Z}\\
     &\iff & n(M+1)-mqj=2k\\
     &\iff & \gcd(n,qj)|2k,
     \end{eqnarray*}
and the assertion follows.
\item We see that $w_{2k+1}=\hat{w}_{2k+1} \iff t(a_{2k+1})=\iota(a_{2(n-k)})$. As $\iota(a_{2(n-k)})=t(a_{2(n-k-1)+1})=t(a_{2(n-k+m(qj-1)-1)+1})$, we have:
     \begin{eqnarray*}
    w_{2k+1}=\hat{w}_{2k+1} & \iff & 2k+1 \equiv 2(n-k+m(qj-1)-1)+1\pmod{2n}, \text{where } m \in \mathbb{N} \\
     &\iff & 2k\equiv 2(n-k+m(qj-1)-1)\pmod{2n}\\
     &\iff & k\equiv n-k+m(qj-1)-1 \pmod{n}\\
     &\iff & n-k+m(qj-1)-1=nM'+k, \text{where } M'\in\mathbb{Z} \\
     &\iff & n(1-M')+m(qj-1)=2k+1\\
     &\iff & \gcd(n,qj-1)|2k+1,
     \end{eqnarray*}
     which proves the assertion.
\end{enumerate}   
\end{proof}

\begin{prop}\label{prop:admissible pants}
    Let $F \in {\Mod}(S_g)$ be an irreducible Type 1 action given by the data set $D_F = (n,0; (c_1,n_1),(c_2,n_2), (c_3,n))$ and realized by the unique hyperbolic polygon $\P_F$. Let $\mathcal{C}$ be as above. Then the length of each curve in $\mathcal{C}$ is bounded by the Ber's constant of $S_g$.
\end{prop}
\begin{proof}
 From the construction of the polygon $\P_{F}$, it follows that the interior angles at vertices $\iota(a_{2i-1})$ and $\iota(a_{2i})$ are equal to $\frac{2\pi}{n_1}$ and $\frac{2\pi}{n_2}$ respectively. Let $O$ denote the centre of $\P_{F}$. Also, let $L_1$ (respectively $L_2$) be the length of the geodesics joining $O$ to $\iota(a_{2i-1})$ (respectively, to $\iota(a_{2i})$). Using basic hyperbolic trigonometry, we see that
\begin{eqnarray*}
  \cosh(L_1) & = & \dfrac{\cos(\frac{\pi}{n_1})\cos(\frac{\pi}{n})+\cos(\frac{\pi}{n_2})}{\sin(\frac{\pi}{n_1})\sin(\frac{\pi}{n})} \text{ and}  \\ 
  \cosh(L_2) & = &\dfrac{\cos(\frac{\pi}{n_2})\cos(\frac{\pi}{n})+\cos(\frac{\pi}{n_1})}{\sin(\frac{\pi}{n_2})\sin(\frac{\pi}{n})}
\end{eqnarray*}
 Let $P_i$ and $Q_i$ be the end points of $\gamma_i$. Consider the isosceles triangles, $\Delta OP_iQ_i$, whose bases are $\gamma_i$. From the symmetries of $P_{\kappa(D)}$, we have 
 $$\angle P_iOQ_i=\frac{2\pi i}{2n}=\frac{i \pi}{n}.$$
Again, from basic hyperbolic trigonometry, we get:
$$\cosh(\ell(\gamma_{i}))= 
\begin{cases} 
\cosh^2(L_2)+\sinh^2(L_2)\cos(\frac{i\pi}{n}), & \text{if } i \text{ is odd, and} \\
\cosh^2(L_1)+\sinh^2(L_1)\cos(\frac{i\pi}{n}), & \text{if } i \text{ is even}.
\end{cases}$$
    This implies that the length of any curve from the collection $\mathcal{C}$ is bounded above by the constant $2\max\{A,B\}$ where 
    \begin{eqnarray*} 
    A & = &\cosh^{-1}\Big(\cosh^2(L_2)+\sinh^2(L_2)\cos(\frac{i\pi}{n})\Big) \text{ and} \\
    B & = &\cosh^{-1}\Big(\cosh^2(L_1)+\sinh^2(L_1)\cos(\frac{i\pi}{n})\Big). 
    \end{eqnarray*}
Finally, the assertion follows from bounds on $\ell(\gamma_{i})$.
\end{proof}
Combining Propositions \ref{prop:pants} and \ref{prop:admissible pants}, we obtain the following theorem.
\begin{theorem}
\label{thm:admiss_pants_decomp}
Let $F \in {\Mod}(S_g)$ be an irreducible Type 1 action given by the data set $D_F = (n,0; (c_1,n_1), (c_2,n_2), (c_3,n))$ and realized by the unique hyperbolic polygon $\P_F$ as mentioned in Theorem \ref{res:1}. Then there exists an admissible pants decomposition for $\P_F$ that is a sub-collection of $\mathcal{C}$ where $\mathcal{C}$ refers to the collection of curves mentioned in Proposition \ref{prop:pants}.
\end{theorem}  

\section{Estimating distances between fixed points of irreducible type 1 actions}
\label{Section:distances_Pants graph}
Let $F \in {\Mod}(S_g)$ be an irreducible Type 1 action given by the data set $D_F 
= (n,0; (c_1,n_1),(c_2,n_2), (c_3,n))$ and realized by the hyperbolic polygon $\P_F
$. In this section, we obtain a combinatorial encoding of a given 
admissible pants decomposition of $S_g$ with respect to $\P_F$, as an ordered $2g$-tuple of positive integers. By realizing the $2g$-tuple as a 
vertex of the pants graph, we then compute distances between two such vertices. Finally, we provide an estimate of the distance between the fixed points of two irreducible Type 1 actions $F_1, F_2 \in {\Mod}(S_g)$ in the moduli space $\M(S_g)$ via the quasi-isometry due to Brock~\cite{JB}. 
\subsection{A combinatorial encoding of an admissible pants decomposition} 
In Section \ref{Section:Admissible_Pants Decomposition} (see Proposition \ref{prop:pants}) $\P_{F}$ was decomposed via an admissible pants decomposition into canonical building blocks comprising cylinders, pairs of pants and four holed spheres, glued across compatible boundary components as shown in the diagram below.
Associating a positive integer to each of the boundary components of these building blocks, which are loops, one can represent these cylinders, pairs of pants and four holed spheres as a pair, triple and quadruple of positive integers respectively. More specifically, besides associating a positive integer to a loop representing a boundary component of these building blocks, we associate positive half integers (i.e. a positive integer divided by $2$) to half loops representing half of a boundary component of these building blocks with the interpretation that two half loops are \textit{closed} if together they form a loop. Thus a boundary component of a building block can be represented either as a positive integer (viewed as a single loop) or as a pair of positive half integers viewed as the union of two closed half loops identified through their end points. As a consequence, one can accordingly represent a building block either as a canonical $k$-tuple of positive integers or a canonical $k'$-tuple comprising positive integers and half integers.  Finally, combining all these tuples corresponding to the canonical building blocks as mentioned in Proposition \ref{prop:pants}, we combinatorially represent an admissible pants decomposition of $\P_{F}$ as a \textit{canonical $g$-tuple}. A precise definition of a canonical $g$-tuple is as follows.

\begin{defn}
A \textit{canonical $g$-tuple} is a surjective map $$f:\{1,2,\cdots,2g\}\rightarrow\{1,2,\cdots,g\},$$ where $f^{-1}(\{i\})$ has exactly two numbers, for all $i$. Moreover, two canonical $g$-tuples $f_1$ and $f_2$ are said to be \textit{equivalent} if at least one of the following conditions hold:
\begin{enumerate}[(i)]
\item $f_2=f_1 \circ \sigma_{0}$, where $\sigma_{0}=(1(2g))(2 (2g-1))\cdots((g-1)(g+1))$ is an element of $\Sigma_{2g}$, the symmetric group on $2g$ letters.
\item $f_2=f_1\circ\phi_1$, where $\phi_{1}=(1,2)$.
\item $f_2=f_1\circ\phi_{2g-1}$, where $\phi_{2g-1}=((2g-1) (2g))$
\item If $f_1(i)=f_1(j)$, for some $i\neq j$, then $f_2(i)=f_2(j)$.
\end{enumerate} 
We denote the equivalent class of a canonical tuple $f$ by $[f]$. 
\end{defn}

For $1 \leq i \leq 2g_1$, let $\phi_i$ denote a transposition of the form $(i \,\,i+1) \in \Sigma_{2g}$. Given two non-equivalent canonical tuples $f_1$ and $f_2$, there exists a unique element $\sigma(f_1,f_2) \in\Sigma_{2g}$ such that $f_2=f_1\circ\sigma(f_1,f_2)$. Given two canonical $g$-tuples, one can define their distance in the pants graph as mentioned below.

\begin{defn}
Let $f_1$ and $f_2$ be canconical $g$-tuples. 
\begin{enumerate}[(i)]
\item The \textit{distance between} $f_1$ and $f_2$, denoted by $d(f_1,f_2)$, is the minimal length of a word in $\{\sigma_i: 1 \leq i \leq 2g_1\}$ that represents $ \sigma(f_1,f_2)$ in $\Sigma_{2g}$. 
\item The \textit{distance between the equivalent classes} of $[f_1]$ and $[f_2]$, denoted by $\mathbb{D}([f_1],[f_2])$, is given by: 
$$\mathbb{D}([f_1],[f_2]):=\min\{d(\tilde{f},\tilde{g}):\tilde{f}\in [f_1],\tilde{g}\in[f_2]\}.$$
\end{enumerate}
\end{defn}  

\begin{prop} \label{pants algorithm}
Let $\P_F$ denote the hyperbolic polygon representing an irreducible Type 1 action on $S_g$. An admissible pants decomposition of $\P_F$ as described in Proposition \ref{prop:pants}, can be represented by a canonical $2g$-tuple. Moreover, any two such representations are equivalent as canonical $2g$-tuples.
\end{prop}    

\begin{proof}
Let us denote these sub-regions by $\mathcal{U}_i$'s. Without loss of generality, one may assume that $a_{11}=a_{2n}$ and $a_{12}=a_1$, and write:
$$\mathcal{U}_1=\left(a_{2n}a_{1}\cdots a_{2Q}a_{2Q+1}  \cdots a_{2(n_1-1)Q}a_{2(n_1-1)Q+1}\right).$$ 
\noindent Thus, we can rearrange the edges corresponding to $\mathcal{U}_1$ in the following fashion:
$$\left(a_{1}a_{2Q}a_{2Q+1} \cdots a_{4Q}a_{4Q+1} \cdots a_{2(n_1-1)Q}a_{2(n_1-1)Q+1}a_{2n}\right).$$
\noindent Following this sequential arrangement of edges, we further rearrange and combine the canonical pieces $R_i$'s accordingly as follows:
$$\left(R_{1}R_{2Q}R_{2Q+1} R_{4Q}R_{4Q+1} R_{6Q} \cdots R_{2(n_1-1)Q}R_{2(n_1-1)Q+1}R_{2n} \right).$$
\noindent We note that there are $2n_1$ number of pieces including identifications such as $R_{2n}=R_1$, $R_{2(n_1-1)Q+1}=R_{2Q}$, etc. In fact, the first half of the pieces are the same as latter half. Now, the following two major cases can occur:
\begin{enumerate}[\textit{Case} 1.]
\item $n_1,n_2 \neq 2:$ Under this case, there are two sub-cases as described below:
\begin{enumerate}[\textit{Case 1}a.]
\item \textit{When $n_1$ is even.} In this case, the sequence of canonical pieces $R_i$'s corresponding to the sub-region $\mathcal{U}_1$ will look like:
$$\left(R_{1}R_{2Q}R_{2Q+1} R_{4Q}R_{4Q+1} R_{6Q} \cdots R_{2(\frac{n_1}{2}-1)Q}R_{n_1Q}\right ),$$ which can be equivalently rearranged as: 
$$\left(R_{1}R^{-1}_{1}R_{2Q+1} R_{2Q+1}^{-1}R_{4Q+1} R_{4Q+1}^{-1} \cdots R_{2(\frac{n_1}{2}-1)Q}R_{n_1Q} \right).$$
Each part of the form $R_{i}R^{-i}_{1}$ in the sequence above represents a cylinder, or a pair of pants, or a four-holed sphere. This provides a sequential arrangement of the building blocks of $\P_{F}$ associated to the sub-region $\mathcal{U}_1$.

Without loss of generality, let us assume that $a_{21}=a_{n_1Q+3}$, $a_{22}=a_{n_1Q+2}$ and then consider the sub-region $\mathcal{U}_2$. Likewise, there will be a tuple of building blocks corresponding to $\mathcal{U}_2$ as illustrated for $\mathcal{U}_1$. Also, the tuple corresponding to $\mathcal{U}_1$ is glued to the tuple corresponding to $\mathcal{U}_2$. Repeating this process for each $\mathcal{U}_i$ ($i=1,2,\cdots k$), and gluing them inductively, we finally get a sequence of building blocks for $\P_{F}$.

\item \textit{When $n_1$ is odd.} Then the sequence of pieces corresponding to the sub-region $\mathcal{U}_1$ will look like:
$$\left(R_{1}R_{2Q}R_{2Q+1} R_{4Q}R_{4Q+1} R_{6Q} \cdots R_{n_1Q}\right ),$$ which can be equivalently rearranged as 
$$\left(R_{1}R^{-1}_{1}R_{2Q+1} R_{2Q+1}^{-1}R_{4Q+1} R_{4Q+1}^{-1} \cdots R_{2(\frac{n_1}{2}-1)Q}R^{-1}_{2(\frac{n_1}{2}-1)Q}\right ).$$

\indent The preceding step leads to an annulus with $2g$ many boundary components, from which the surface $S_g$ is obtained by gluing a handle to the boundary components belonging to the same equivalence class imposed by the side pairing relations of $\P_{F}$. We observe that, here $g$ handles correspond to $g$ pairwise disjoint non-essential curves. We also get curves between each pair of adjacent boundaries, except for the first and the last boundary components. This provides a new set of curves, disjoint from the previously obtained $g$ curves. Thus we have $g+2g-3=3g-3$ disjoint curves which together represent a pants decomposition of $S_g$ which is also vertex of the pants graph. By construction, it is an admissible pants decomposition as well. Eventually, each of these canonical $2g$-tuple of integers represents a vertex of the pants graph. It might happen that two canonical $2g$-tuples represent the same vertex, and in such cases, it turns out that those tuples are equivalent.
\end{enumerate}

\item One of $n_1$ and $n_2$ is equal to $2$. 
In this case, from Theorem \ref{res:1}, it follows that there are two possibilities (see Theorem \ref{res:1}):

\begin{enumerate}[\textit{Case 2}a.]
\item $\P_F$ is a regular hyperbolic $4g$-gon with opposite side pairing, i.e. $a_i^{-1} \sim a_{2g+i}$, for all $1\leq i \leq 2g$.  It will have exactly one vertex after gluing. Let's denote it by $v$. In this case, there are $2g$ canonical pieces, namely, $R_1, \ldots,R_{2g}$ arranged in the following sequence (see Figure \ref{fig:poly_part}): 
$$\left(R_{g}R_{g+1}R_{g-1} R_{g+2}\cdots R_{2}R_{2g-1}R_1R_{2g}\right ).$$
\item  $\P_F$ is a regular hyperbolic $(4g+2)$-gon with opposite side pairing, i.e. $a_i^{-1} \sim a{2g+1+i}$, for all $1\leq i \leq 2g+1$.  There will be two distinct vertices after gluing, say, $v_1, v_2$. In this case, there are $2g+1$ canonical pieces, namely, $R_1, \ldots,R_{2g+1}$. Note that in this case, $R_{g+1}=R_{g+1}^{-1}$ and therefore, $R_{g+1}$ will form a cylinder. Thus the canonical pieces are arranged in the following sequence (as illustrated in Figure \ref{fig:poly_part}): 
$$\left(R_{g+1}R_{g}R_{g+2}R_{g-1}R_{g+3}\cdots R_{2}R_{2g}R_1R_{2g+1}\right ).$$
\end{enumerate}
\end{enumerate}
This completes the proof.
\end{proof}
\subsection{Estimating the distance in the Moduli space}
Let $F_1,F_2\in \Mod(S_g)$ be irreducible type-1 actions on $S_g$ whose respective unique fixed points in the moduli space are given by the hyperbolic polygons $\mathcal{P}_{F_1}$ and $\mathcal{P}_{F_2}$.  Also, let $[\mathcal{P}_{F_1}]$ and $[\mathcal{P}_{F_2}]$ denote the orbit spaces of $\mathcal{P}_{F_1}$ and $\mathcal{P}_{F_2}$ respectively, for the $\Mod(S_g)$-action on $\Teich(S_g)$. Let
$$\hat{d}_{WP}([\mathcal{P}_{F_1}],[\mathcal{P}_{F_2}]):=\min\{d_{WP}(\mathcal{X}_1,\mathcal{X}_2)| \mathcal{X}_1\in[\mathcal{P}_{F_1}], \mathcal{X}_2\in[\mathcal{P}_{F_2}] \}.$$
The following theorem provides an estimate of the Weil-Petersson distance $\hat{d}_{WP}([\mathcal{P}_{F_1}],[\mathcal{P}_{F_2}])$ using the result of Brock \cite{JB}. 
\begin{theorem}
Let $F_1$ and $F_2$ be two irreducible type-1 actions of $\Mod(S_g)$. Also, let $\mathcal{P}_{F_1}$ and $\mathcal{P}_{F_2}$ be the hyperbolic polygons as defined in Theorem \ref{res:1}. Suppose, $f_1$ and $f_2$ be the canonical tuples associated to $\mathcal{P}_{F_1}$ and $\mathcal{P}_{F_2}$ respectively. Then
$$\hat{d}_{WP}([\mathcal{P}_{F_1}],[\mathcal{P}_{F_2}])\leq K\mathbb{D}([f_1],[f_2])+\epsilon,$$
where $K$ and $\epsilon$ are the quasi-isometric constants from Brock's result.
\end{theorem}
\begin{proof}
  Brock's result implies that, for any two points $\mathcal{X}_1,\mathcal{X}_2\in Teich(S_g)$, 
  $$d_{WP}(\mathcal{X}_1,\mathcal{X}_2)\leq Kd_{\mathcal{P}}(P_{\mathcal{X}_1},P_{\mathcal{X}_2})+\epsilon,$$ 
 where $P_{\mathcal{X}_1}$ and $P_{\mathcal{X}_2}$ denote two admissible pants decompositions corresponding to $\mathcal{X}_1$ and $\mathcal{X}_2$ respectively (Here $d_{\mathcal{P}}$ is the distance function of the \textit{Pants complex}). Let $P_{f_1}$ and $P_{f_2}$ denote the admissible pants decompositions corresponding to the \textit{canonical tuples} $f_1$ and $f_2$ respectively as constructed in Proposition \ref{pants algorithm}. Then it follows that $d_{\mathcal{P}}(P_{f_1},P_{f_2})\leq \mathbb{D}([f_1],[f_2])$. Hence by applying Brock's result, we have 
$$\hat{d}_{WP}([\mathcal{P}_{F_1}],[\mathcal{P}_{F_2}])\leq d_{WP}(\mathcal{X}_{F_1},\mathcal{X}_{F_2})\leq Kd_{\mathcal{P}}(P_{f_1},P_{f_2})+\epsilon \leq K\mathbb{D}([f_1],[f_2])+\epsilon$$  
\end{proof}
   
\bibliographystyle{plain} 
\bibliography{Article_KAS}
\end{document}